\documentclass[3p]{article}
\usepackage{amsfonts,amsmath,amssymb,amsthm,stmaryrd,hyperref}
\usepackage[francais,english]{babel}
\usepackage[T1]{fontenc}
\usepackage[ansinew]{inputenc}  % si les car sont codes windows
\usepackage{graphics}
\usepackage{graphicx}
\usepackage{tikz}
\usepackage{multicol}
\usetikzlibrary{patterns}
\usepackage{amssymb,subfigure,lineno,url}
\usepackage{float}
\usepackage[draft]{showlabels}
\newtheorem{theorem}{Theorem}%[section]
\newtheorem{definition}{Definition}%[section]
\newtheorem{lemma}{Lemma}%[section]
\newtheorem{proposition}{Proposition}%[section]
\newtheorem{remark}{ Remark}%[section]

\begin{document}
\title{Decay of mass for a semilinear heat equation with mixed local-nonlocal operators}

\author{Mokhtar Kirane \footnote{\noindent 
Department of Mathematics, College of Computing and Mathematical Sciences, Khalifa University, P.O. Box: 127788, Abu Dhabi, UAE; mokhtar.kirane@ku.ac.ae.
 \newline \indent $\,\,{}^{a}$  College of Engineering and Technology, American University of the Middle East, Kuwait; ahmad.fino@aum.edu.kw.
  \newline \indent $\,\,{}^{b}$ Lebanese University, Faculty of Sciences, Beirut, Lebanon; a.ayoub94@gmail.com. }, Ahmad Z. Fino$^{\,a}$, Alaa Ayoub$^{\,b}$}

\date{}
\maketitle

\begin{abstract}
In this paper, we are concerned with the Cauchy problem for the reaction-diffusion equation
$\partial_t u+t^\beta\mathcal{L} u= - h(t)u^p$ posed on $\mathbb{R}^N$, driven by the mixed local-nonlocal operator $\mathcal{L}=-\Delta+(-\Delta)^{\alpha/2}$, $\alpha\in(0,2)$,
and supplemented with a nonnegative integrable initial data, where $p>1$, $\beta\geq 0$, and $h:(0,\infty)\to(0,\infty)$ is a locally integrable function. We study the large time behavior of non-negative solutions and show that the nonlinear term determines the large time asymptotic for
$p\leq 1+{\alpha}/{N(\beta+1)},$ while the classical/anomalous diffusion effects win if $p>1+{\alpha}/{N(\beta+1)}$.
\end{abstract}

\maketitle



%=================================================================================={Introduction}
\section{Introduction}
Let $\mathcal{L}$ be the mixed local-nonlocal operator $\mathcal{L}=-\Delta+(-\Delta)^{\alpha/2}$, where $(-\Delta)^{\alpha/2}$ stands for the fractional Laplacian of order $0<\alpha<2$. We study the asymptotic behavior of the solutions of the
following Cauchy problem for the  following reaction-diffusion equation driven by both the classical and anomalous diffusion
\begin{equation}\label{eq}
\left\{\begin{array}{ll}\partial_t u+t^{\beta}\mathcal{L} u=-h(t)u^p,&\qquad x\in\mathbb{R}^N,\,t>0,\\\\
 u(x,0)= u_0(x)\geq0,&\qquad x\in\mathbb{R}^N,\\
 \end{array}\right.
\end{equation}
where $p>1$, $\beta\geq0$, $h:(0,\infty)\to(0,\infty)$,  $h\in L^1_{loc}(0,\infty)$, and
$$u_0\in L^1(\mathbb{R}^N)\cap C_0(\mathbb{R}^N).$$
The space $C_0(\mathbb{R}^n)$ denotes the space of all continuous functions tending to zero at infinity.\\

Large time behavior of solutions has been extensively studied in the literature. Concerning the purely local or nonlocal case $(-\Delta)^{s/2}$, $0<s\leq 2$,  the authors in \cite{Finokarch} have proven, for $\beta=0$, and $h\equiv 1$, that the critical exponent for the large time behavior of solutions of \eqref{eq} is $p=1+s/N$. In fact, the authors studied the decay of the mass  $\displaystyle M(t)= \int_{\mathbb{R}^N}u(x,t)\,dx$ of the solutions and proved the following:\\

\noindent $\bullet$ if $p\leq 1+s/N$, then $\lim\limits_{t\rightarrow\infty}M(t)=0$;\\
$\bullet$  if $p> 1+s/N$, then there exists $M_\infty\in(0,\infty)$ such that $\lim\limits_{t\rightarrow\infty}M(t)=M_\infty>0$.\\

Recently, the study of qualitative properties of solutions to elliptic/parabolic differential equations driven by the mixed local-nonlocal operator $\mathcal{L}=-\Delta+(-\Delta)^{\alpha/2}$, $\alpha\in(0,2)$, has attracted a great amount of attention (see \cite{Biagi,Biagi2,Biagi3,Bonforte2,Bonforte3,Pezzo,DeFilippis,Garain1,Garain2,Garain3,Garain4}). One of the main reasons for considering this type of operator is that  $\mathcal{L}$ appears naturally in applied sciences, to study the role of the impact caused by a local and a nonlocal change in a physical phenomenon. These type of operators also have applications in probability. Indeed, they are associated with the superposition of different types of stochastic processes, such as classical random walks and L\'evy flights. Essentially, when a particle can follow either process with certain probabilities, the resulting limiting diffusion equation can be described by an operator of the form $\mathcal{L}$. In addition, the mixed operators are applied to model a range of phenomena in science, including the analysis of optimal animal foraging strategies (see, e.g. \cite{Dipierro1,Dipierro2} and references therein). Regarding the asymptotic behavior of solutions including local/nonlocal operators $(-\Delta)^{s/2}$, $0<s\leq 2$, we refer the reader to \cite{BashirKirane,Finokarch,Jleli,KiraneQaf,KiraneQaf2}.\\

When the nonlinear term $-u^p$ is replaced by a power-like source term $u^p$, 
$$
\left\{\begin{array}{ll}\partial_t u+\mathcal{L} u=u^p,&\qquad x\in\mathbb{R}^N,\,t>0,\\\\
 u(x,0)= u_0(x)\geq0,&\qquad x\in\mathbb{R}^N,\\
 \end{array}\right.
$$
a full study concerning the global existence, blow-up, blow-up rate, as well as the asymptotic behavior of solutions of \eqref{eq}, have been recently studied by \cite{Biagi,Pezzo}. Note that, the operator $\mathcal{L}$ affects negatively to the scaling property, and the Fujita exponent is given by the fractional Laplacian i.e. $p_F=1+\alpha/N$. Namely, they proved that any nonnegative nontrivial solution blows up in finite time if $p\leq p_F$ while there exists a small data global in time solution when $p>p_F$. \\

Our goal is to contribute to the theory of large time behavior of solutions by generalizing the work of \cite{Finokarch,Jleli} to the mixed local-nonlocal operator $\mathcal{L}$. Our results can be summarized as follows. A unique non-negative solution of the Cauchy problem \eqref{eq} exists globally in time.  Hence, we  study the
decay properties of the mass 
$$ M(t)= \int_{\mathbb{R}^N}u(x,t)\,dx,$$
 of the solutions $u=u(x,t)$ to problem \eqref{eq}. We prove that $ \lim\limits_{t\to\infty}M(t)=M_\infty >0$ for $p>1+{\alpha}/{N(\beta+1)}$ (Theorem \ref{decay}, below), while $M(t)$ tends to zero as $t\to\infty$ if $1<p\leq 1+\alpha/N(\beta+1)$ (Theorem \ref{convto0}). We point out that problem \eqref{eq}, when $\beta=0$, and $h\equiv 1$, behaves like the problem with purely nonlocal operator $\mathcal{L}=(-\Delta)^{\alpha/2}$, $0<\alpha< 2$. The proof of Theorem \ref{decay} is based on the $L^p-L^q$ estimates of solutions as well as the comparison principle, while for Theorem \ref{convto0}, our proof approach is based on the method of nonlinear capacity estimates or the so-called rescaled test function method. The nonlinear capacity method was introduced to prove the non-existence of global solutions by  Baras and Pierre \cite{Baras-Pierre}, then used by Baras and Kersner in \cite{Baras}; later on, it was developed by Zhang in \cite{Zhang} and Mitidieri and Pohozaev in \cite{19}.  It was also used by Kirane et al.  in \cite{Kirane-Guedda,KLT},  and Fino et al in \cite{Fino1,Fino3,Fino4,Fino5}. \\

The paper is organized as follows. In Section \ref{sec2}, we present some definitions, terminologies, and preliminary results concerning the fractional Laplacian, the operator $\mathcal{L}$ and its heat kernel. Section \ref{sec3} is devoted to present our main results while their proofs are given in Sections \ref{sec4} and \ref{sec5}.\\

%%%%%%%%%%%%%%%%%%%%%%%%%%%%%%%%%%%%%%%%%%%%%%%%%%%%%%%%%%%%%%%%%%%%%%%

\section{Preliminaries}\label{sec2}
We first recall the definition and some properties related to the fractional Laplacian needed to prove Theorem \ref{convto0}.
\begin{definition}\cite{Silvestre}\label{def1}
Let $u\in\mathcal{S}$  be the Schwartz space of rapidly decaying $C^\infty$ functions in $\mathbb{R}^N$ and $s \in (0,1)$. The fractional Laplacian $(-\Delta)^s$ in $\mathbb{R}^N$ is a non-local operator given by
\begin{eqnarray*}
 (-\Delta)^s v(x)&:=& C_{N,s}\,\, p.v.\int_{\mathbb{R}^N}\frac{v(x)- v(y)}{|x-y|^{N+2s}}\,dy\\\\
&=&\left\{\begin{array}{ll}
\displaystyle C_{N,s}\,\int_{\mathbb{R}^N}\frac{v(x)- v(y)}{|x-y|^{N+2s}}\,dy,&\quad\hbox{if}\,\,0<s<1/2,\\
{}\\
\displaystyle C_{N,s}\,\int_{\mathbb{R}^N}\frac{v(x)- v(y)-\nabla v(x)\cdotp(x-y)\mathcal{X}_{|x-y|<\delta}(y)}{|x-y|^{N+2s}}\,dy,\quad\forall\,\delta>0,&\quad\hbox{if}\,\,1/2\leq s<1,\\
\end{array}
\right.
\end{eqnarray*}
where $p.v.$ stands for Cauchy's principal value, and $C_{N,s}:= \frac{s\,4^s \Gamma(\frac{N}{2}+s)}{\pi^{\frac{N}{2}}\Gamma(1-s)}$.
\end{definition}

We are rarely going to use the fractional Laplacian operator in the Schwartz space; it can be extended to less regular functions as follows. For $s \in (0,1)$, $\varepsilon>0$, let
\begin{eqnarray*}
L_{s,\varepsilon}(\Omega)&:=&\left\{\begin{array}{ll}
\displaystyle L_s(\mathbb{R}^N)\cap C^{0,2s+\varepsilon}(\Omega)&\quad\hbox{if}\,\,0<s<1/2,\\
{}\\
\displaystyle L_s(\mathbb{R}^N)\cap C^{1,2s+\varepsilon-1}(\Omega),&\quad\hbox{if}\,\,1/2\leq s<1,\\
\end{array}
\right.
\end{eqnarray*}
where $\Omega$ is an open subset of $\mathbb{R}^N$, $C^{0,2s+\varepsilon}(\Omega)$ is the space of $2s+\varepsilon$- H\"{o}lder continuous functions on $\Omega$,  $C^{1,2s+\varepsilon-1}(\Omega)$ the space of functions of $C^1(\Omega)$ whose first partial
derivatives are H\"{o}lder continuous with exponent $2s+\varepsilon-1$, and
$$L_s(\mathbb{R}^N)=\left\{u:\mathbb{R}^N\rightarrow\mathbb{R}\quad\hbox{such that}\quad \int_{\mathbb{R}^N}\frac{u(x)}{1+|x|^{N+2s}}\,dx<\infty\right\}.$$

\begin{proposition}\label{Frac}\cite[Proposition~2.4]{Silvestre}${}$\\
Let $\Omega$ be an open subset of $\mathbb{R}^N$, $s \in (0,1)$, and $f\in L_{s,\varepsilon}(\Omega)$ for some $\varepsilon>0$. Then $(-\Delta)^sf$ is a continuous function in $\Omega$ and $(-\Delta)^sf(x)$ is given by the pointwise formulas of Definition \ref{def1} for every $x\in\Omega$.
\end{proposition} 
\noindent{\bf Remark:} A simple sufficient condition for function $f$ to satisfy the conditions in Proposition \ref{Frac} is that $f\in  L^1_{loc}(\mathbb{R}^N)\cap C^{2}(\Omega)$.\\

\begin{lemma}\label{lemma3}\cite{Bonforte}
Let $\langle x\rangle:=(1+|x|^2)^{1/2}$, $x\in\mathbb{R}^N$, $s \in (0,1]$, $d\geq 1$, and $q_0>N$. Then 
$$\langle x\rangle^{-q_0}\in  L^\infty(\mathbb{R}^N)\cap C^{\infty}(\mathbb{R}^N),\qquad\partial_x^2\langle x\rangle^{-q_0}\in L^\infty(\mathbb{R}^N),$$
 and 
$$
\left|(-\Delta)^s\langle x\rangle^{-q_0}\right|\lesssim \langle x\rangle^{-N-2s}.
$$
\end{lemma}

\begin{lemma}\label{lemma4}
Let $\psi$ be a smooth function satisfying $\partial_x^2\psi\in L^\infty(\mathbb{R}^N)$. For any $R>0$, let $\psi_R$ be a function defined by
$$ \psi_R(x):= \psi(R^{-1} x) \quad \text{ for all } x \in \mathbb{R}^N.$$
Then, $(-\Delta)^s (\psi_R)$,  $s \in (0,1]$,  satisfies the following scaling properties:
$$(-\Delta)^s \psi_R(x)= R^{-2s}(-\Delta)^s\psi(R^{-1} x), \quad \text{ for all } x \in \mathbb{R}^N. $$
\end{lemma}

\begin{lemma}\label{lemma5} Let $R>0$, $p>1$, $0<\alpha <2$, $N\geq1$, and $N<q_0<N+\alpha p$. Then, the following estimate holds
\begin{equation}\label{10}
\int_{\mathbb{R}^N}(\Phi_R(x))^{-1/(p-1)}\,\big|\mathcal{L}\Phi_R(x)\big|^{p/(p-1)}\, dx\lesssim R^{-\frac{2 p}{p-1}+N}+R^{-\frac{\alpha p}{p-1}+N},
\end{equation}
where $\Phi_R(x)=\langle {x}/{R}\rangle^{-q_0}=(1+|x/R|^2)^{-q_0/2}$.
\end{lemma}
\begin{proof}  Let $\tilde{x}=x/R$; by Lemma \ref{lemma4} we have $(-\Delta)^s\Phi_R(x)=R^{-2s}(-\Delta)^s\Phi_R(\tilde{x})$, for $s\in\{1,\alpha/2\}$. Therefore, using Lemma \ref{lemma3}, we conclude that
\begin{eqnarray*}
&{}&\int_{\mathbb{R}^N}(\Phi_R(x))^{-1/(p-1)}\,\big|\mathcal{L} \Phi_R(x)\big|^{p/(p-1)}\, dx\\
&{}&\lesssim\int_{\mathbb{R}^N}(\Phi_R(x))^{-1/(p-1)}\,\big|(-\Delta) \Phi_R(x)\big|^{p/(p-1)}\, dx+\int_{\mathbb{R}^N}(\Phi_R(x))^{-1/(p-1)}\,\big|(-\Delta)^{\alpha/2} \Phi_R(x)\big|^{p/(p-1)}\, dx\\
&{}&\lesssim R^{-\frac{2p}{p-1}+N}\int_{\mathbb{R}^N}\langle \tilde{x}\rangle^{\frac{q_0}{p-1}-\frac{(N+2)p}{p-1}}\, d \tilde{x}+R^{-\frac{\alpha p}{p-1}+N}\int_{\mathbb{R}^N}\langle \tilde{x}\rangle^{\frac{q_0}{p-1}-\frac{(N+\alpha)p}{p-1}}\, d \tilde{x}\\
&{}&\lesssim R^{-\frac{2p}{p-1}+N}+R^{-\frac{\alpha p}{p-1}+N},
\end{eqnarray*}
where we have used the fact that $q_0<N+\alpha p<N+2p$ implies $\frac{(N+2)p}{p-1}-\frac{q_0}{p-1}>N$ and $\frac{(N+\alpha)p}{p-1}-\frac{q_0}{p-1}>N$.
\end{proof}

In order to introduce the notion of  the mild solution of our problem, and to prove the main results especially Theorem \ref{decay}, we need to give a full review about the heat semigroup and the associated heat kernel of  the operator $-\mathcal{L}=\Delta-(-\Delta)^{\alpha/2}$, $\alpha\in (0,2)$ (see e.g. \cite{Song}). The comparison principle is needed as well.\\
We start by an alternative expression of the fractional Laplacian operator $(-\Delta)^{\alpha/2}$ via the Fourier transform 
$$
\begin{array}{llll}
(-\Delta)^{\alpha/2}:&H^\alpha(\mathbb{R}^N)\subseteq L^2(\mathbb{R}^N)&\longrightarrow& L^2(\mathbb{R}^N)\\
{}&\qquad \quad u&\longmapsto& (-\Delta)^{\alpha/2}(u)=\mathcal{F}^{-1}(|\xi|^\alpha\mathcal{F}(u))\\
\end{array}
$$
where 
$$H^\alpha(\mathbb{R}^N)=D\left((-\Delta)^{\alpha/2}\right)=\{u\in L^2(\mathbb{R}^N);\,\, |\xi|^\alpha\mathcal{F}(u)\in L^2(\mathbb{R}^N)\},\qquad s>0.$$
 As a consequence of the fact that $(-\Delta)^{\alpha/2}$ is a positive definite self-adjoint operator on the Hilbert space $L^2(\mathbb{R}^N)$, is that the operator $-(-\Delta)^{\alpha/2}$, e.g. Yosida \cite{Yosi} or \cite[Theorem~4.9]{Grigoryan}, generates a strongly continuous semigroup $T(t):=e^{-t(-\Delta)^{\alpha/2}}$ on $L^2(\mathbb{R}^N)$. It holds $T(t)v=P_\alpha(t)\ast v$, where $P_\alpha$ is
the fundamental solution of the fractional diffusion equation $u_t+(-\Delta)^{\alpha/2}u=0$, represented via the Fourier transform by
\begin{equation}\label{FS}
P_\alpha(t)(x):=P_\alpha(x,t)=\frac{1}{(2\pi)^{N/2}}\int_{\mathbb{R}^N}e^{ix.\xi-t|\xi|^\alpha}\,d\xi.
\end{equation}
It is well-known that this function, also known  as the heat kernel of $-(-\Delta)^{\alpha/2}$, satisfies
\begin{equation}\label{P_1}
    P_\alpha(1)\in L^\infty(\mathbb{R}^N)\cap
L^1(\mathbb{R}^N),\quad
P_\alpha(x,t)\geq0,\quad\int_{\mathbb{R}^N}P_\alpha(x,t)\,dx=1,
\end{equation}
\noindent for all $x\in\mathbb{R}^N$ and $t>0.$ Hence, using Young's inequality for the convolution
and the following self-similar form $P_\alpha(x,t)=t^{-N/\alpha}P_\alpha(xt^{-1/\alpha},1)$, see e.g. \cite{Miao}, we have
\begin{equation}\label{P_2}
\|P_\alpha(t)\ast v\|_q\;\leq \;Ct^{-\frac{N}{\alpha}(\frac{1}{r}-\frac{1}{q})}\|v\|_r,
\end{equation}
\begin{equation}\label{P_3}
\|\nabla P_\alpha(t)\|_q\;\leq \;Ct^{-\frac{N}{\alpha}(1-\frac{1}{q})-\frac{1}{\alpha}},
\end{equation}
for all $v\in L^r(\mathbb{R}^N)$ and  all $1\leq r\leq q\leq\infty,$ $t>0$.\\
Using $H^2(\mathbb{R}^N)\subseteq H^\alpha(\mathbb{R}^N)$, one can define the mixed local-nonlocal operator by
$$
\begin{array}{llll}
\mathcal{L}:&H^2(\mathbb{R}^N)\subseteq L^2(\mathbb{R}^N)&\longrightarrow& L^2(\mathbb{R}^N)\\
{}&\qquad \quad u&\longmapsto& \mathcal{L}(u)=-\Delta u+(-\Delta)^{\alpha/2}(u)\\
\end{array}
$$
which, naturally, implies that the operator $-\mathcal{L}$ generates a strongly continuous semigroup of contractions $\{S(t):=e^{-t\mathcal{L}}\}_{t\geq 0}$ on $L^2(\mathbb{R}^N)$. It holds $S(t)v=E_\alpha(t)\ast v$, where $E_\alpha(t)=P_2(t)\ast P_\alpha(t)$ is the heat kernel of the operator $-\mathcal{L}$, and $P_2(t)$ is the fundamental solution of the heat  equation $u_t-\Delta u=0$, represented by
$$P_2(t)=P_2(t,x)=\frac{1}{(4\pi t)^{N/2}}e^{-\frac{|x|^2}{4t}}.$$
Note that,  see \cite[Proposition~48.4]{souplet}, the function $P_2(t)$ is the heat kernel of $\Delta$ and satisfies the properties \eqref{P_1}, \eqref{P_2}, and \eqref{P_3} by replacing $\alpha$ by $2$.\\
It is clear that $E_\alpha(t)$ is the fundamental solution of the diffusion equation $u_t + \mathcal{L}u=0$, and satisfies the following properties.

\begin{lemma}\label{Property1}
The heat kernel $E_\alpha(t)$ satisfies the following properties
\begin{enumerate}
  \item[$(i)$] $E_\alpha \in C^{\infty}(\mathbb{R}^+\times\mathbb{R}^N)$ and $E_\alpha(t)\geq 0$, for all $x\in\mathbb{R}^N$ and $t>0$,
  \item[$(ii)$] For all $t>0$, we have 
  $$\int_{\mathbb{R}^N}E_\alpha(t,x)dx=1.$$
   \item[$(iii)$] For every $t,\tau>0$, we have 
  $$E_\alpha(t)\ast E_\alpha(\tau)=E_\alpha(t+\tau).$$
  \end{enumerate}
\end{lemma}
\begin{lemma}\label{Property2}
The heat kernel $E_\alpha(t)$ satisfies the following properties
\begin{itemize}
   \item For every $v\in L^r(\mathbb{R}^N)$ and  all $1\leq r\leq q\leq\infty$, $t>0$, we have
  \begin{equation}\label{Pr1}
  \|E_\alpha(t)\ast v\|_q\,\leq \;C\min\left\{t^{-\frac{N}{2}(\frac{1}{r}-\frac{1}{q})},\,t^{-\frac{N}{\alpha}(\frac{1}{r}-\frac{1}{q})}\right\}\|v\|_r.
\end{equation}
   \item For all $1\leq q\leq \infty$, we have 
    \begin{equation}\label{Pr2}
    \|\nabla E_\alpha(t)\|_q\,\leq \;C\min\left\{t^{-\frac{N}{2}(1-\frac{1}{q})-\frac{1}{2}},\,t^{-\frac{N}{\alpha}(1-\frac{1}{q})-\frac{1}{\alpha}}\right\}.
  \end{equation}
  \end{itemize}
\end{lemma}
\begin{proof} For every $1\leq r\leq q\leq\infty$ and $t>0$, using Young's inequality for the convolution and the fact that $\|P_i(t)\|_1=1$, $i=2,\alpha$, we have
\begin{eqnarray}\label{Pr4}
\|E_\alpha(t) \ast v\|_q&=&\|P_2(t) \ast \left(P_\alpha(t)\ast v\right)\|_q\nonumber\\
&\leq&C\,t^{-\frac{N}{2}(\frac{1}{r}-\frac{1}{q})}\|P_\alpha(t)\ast v\|_r\nonumber\\
&\leq&C\,t^{-\frac{N}{2}(\frac{1}{r}-\frac{1}{q})}\|P_\alpha(t)\|_1\|v\|_r\nonumber\\
&=&C\,t^{-\frac{N}{2}(\frac{1}{r}-\frac{1}{q})}\|v\|_r,
\end{eqnarray}           
  and
   \begin{eqnarray}\label{Pr5}
\|E_\alpha(t) \ast v\|_q&=&\|P_\alpha(t) \ast \left(P_2(t)\ast v\right)\|_q\nonumber\\
&\leq&C\,t^{-\frac{N}{\alpha}(\frac{1}{r}-\frac{1}{q})}\|P_2(t)\ast v\|_r\nonumber\\
&\leq&C\,t^{-\frac{N}{\alpha}(\frac{1}{r}-\frac{1}{q})}\|P_2(t)\|_1\|v\|_r\nonumber\\
&=&C\,t^{-\frac{N}{\alpha}(\frac{1}{r}-\frac{1}{q})}\|v\|_r.
\end{eqnarray}                
Combining \eqref{Pr4} and \eqref{Pr5}, \eqref{Pr1} is obtained. Similarly, to get \eqref{Pr2}, we have
$$
\|\nabla E_\alpha(t)\|_q=\|\nabla(P_2(t)) \ast P_\alpha(t)\|_q\leq\|\nabla(P_2(t))\|_q \|P_\alpha(t)\|_1\leq C\,t^{-\frac{N}{2}(1-\frac{1}{q})-\frac{1}{2}},
$$            
  and
  $$
\|\nabla E_\alpha(t)\|_q=\|\nabla(P_\alpha(t)) \ast P_2(t)\|_q\leq\|\nabla(P_\alpha(t))\|_q \|P_2(t)\|_1\leq C\,t^{-\frac{N}{\alpha}(1-\frac{1}{q})-\frac{1}{\alpha}}.
$$
\end{proof}
\begin{lemma}\label{Taylor}
Let $g\in L^1(\mathbb{R}^N)$ and put $\displaystyle M_g=\int_{\mathbb{R}^N}g(x)\,dx$. We have 
\begin{equation}\label{TaylorInequality1}
\lim\limits_{t\to\infty}\|E_\alpha(t)\ast g-M_gE_\alpha(t)\|_1=0.
\end{equation}
If, in addition, $xg(x) \in L^1(\mathbb{R}^N)$, then
\begin{equation}\label{TaylorInequality2}
\|E_\alpha(t)\ast g-M_gE_\alpha(t)\|_1\leq C\min\{t^{-1/2},\,t^{-1/\alpha}\}\|xg(x)\|_1,\qquad\hbox{for all}\,\,t>0.
\end{equation}
\end{lemma}
\begin{proof}
We adapt the technique used in \cite[Proposition~48.6]{souplet}. We first establish \eqref{TaylorInequality2} by supposing $g\in L^1(\mathbb{R}^N,\,(1+|x|)\,dx)$. Using Taylor's expansion and Fubini's theorem, we have
\begin{eqnarray*}
\|E_\alpha(t)\ast g-M_gE_\alpha(t)\|_1&=&\left\|\int_{\mathbb{R}^N}\left(E_\alpha(t,x-y)-E_\alpha(t,x)\right)g(y)\,dy\right\|_1\\
&=&\left\|\int_0^1\int_{\mathbb{R}^N}\nabla E_\alpha(t,x-\theta y) yg(y)\,dy\,d\theta\right\|_1\\
&\leq&\int_0^1\int_{\mathbb{R}^N}\left\|\nabla E_\alpha(t,x-\theta y)\right\|_1 yg(y)\,dy\,d\theta\\
&\leq&C\min\left\{t^{-\frac{1}{2}},\,t^{-\frac{1}{\alpha}}\right\}\int_0^1\int_{\mathbb{R}^N}yg(y)\,dy\,d\theta\\
&\leq&C\min\{t^{-1/2},\,t^{-1/\alpha}\}\|xg(x)\|_1,
\end{eqnarray*}
where we have use Minkowski's inequality and \eqref{Pr2}. Let us next prove \eqref{TaylorInequality1}; fix $g\in L^1(\mathbb{R}^N)$ and pick a sequence $\{g_j\}\in \mathcal{D}(\mathbb{R}^N)$ such that $g_j\rightarrow g$ in $L^1(\mathbb{R}^N)$. For each $j$, using the fact that $\|E_\alpha(t)\|_1=1$, we have
\begin{eqnarray*}
\|E_\alpha(t)\ast g-M_gE_\alpha(t)\|_1&\leq&\|E_\alpha(t)\ast g-E_\alpha(t)\ast g_j\|_1+\|E_\alpha(t)\ast g_j-M_{g_j}E_\alpha(t)\|_1+\|M_{g_j}E_\alpha(t)-M_gE_\alpha(t)\|_1\\
&\leq&\|g-g_j\|_1\|E_\alpha(t)\|_1+\|E_\alpha(t)\ast g_j-M_{g_j}E_\alpha(t)\|_1+|M_{g_j}-M_g|\|E_\alpha(t)\|_1\\
&\leq&2\|g-g_j\|_1+C\min\{t^{-1/2},\,t^{-1/\alpha}\}\|xg_j(x)\|_1.
\end{eqnarray*}
By \eqref{TaylorInequality2}, it follows that
$$\limsup_{t\rightarrow\infty}\|E_\alpha(t)\ast g-M_gE_\alpha(t)\|_1\leq 2\|g-g_j\|_1,$$
and the conclusion follows by letting $j\rightarrow\infty$.
\end{proof}
Before we end up the section, an important lemma will be used to provide  the fundamental solution of the diffusion equation $u_t + t^{\beta}\mathcal{L}u=0$.
\begin{lemma}\label{Property3}
Let us consider the homogeneous system
$$
\left\{\begin{array}{ll}
\displaystyle \partial_t u+t^{\beta}\mathcal{L} u=0,&\quad x\in\mathbb{R}^N,t>0,\\\\
 u(x,0) = u_0(x),&\quad x\in\mathbb{R}^N.\\
\end{array}\right.
$$
For every $t\geq t_0\geq 0$, it follows that
$$u(t)=E_\alpha\left(\frac{t^{\beta+1}}{\beta+1}-\frac{t_0^{\beta+1}}{\beta+1}\right)\ast u(t_0).$$
\end{lemma}
\begin{proof} Multiply the homogeneous equation by $t^{-\beta}$, we get $t^{-\beta}\partial_t u+\mathcal{L} u=0$. Let $v(x,t):=u(x,t+t_0)$, then
$$(t+t_0)^{-\beta}\partial_t v(x,t)+\mathcal{L} v(x,t)=0.$$
By considering the change of variable
$$\tau=\frac{(t+t_0)^{\beta+1}}{\beta+1}-\frac{t_0^{\beta+1}}{\beta+1},$$
and denoting
$$\widetilde{v}(x,\tau)=v(x,t),$$
we obtain
$$\partial_\tau \widetilde{v}(x,\tau)+\mathcal{L}\widetilde{v}(x,\tau)=0,$$
which yields to
$$\widetilde{v}(x,\tau)=E_\alpha(\tau)\ast \widetilde{v}(x,0),$$
i.e.
$$v(x,t)=E_\alpha\left(\frac{(t+t_0)^{\beta+1}}{\beta+1}-\frac{t_0^{\beta+1}}{\beta+1} \right)\ast v(x,0).$$
As $v(x,0)=u(x,t_0)$, we conclude that 
$$u(x,t)=v(x,t-t_0)=E_\alpha\left(\frac{t^{\beta+1}}{\beta+1}-\frac{t_0^{\beta+1}}{\beta+1} \right)\ast u(x,t_0).$$
\end{proof}
\begin{definition}[Mild solution]
Let $u_0\in  C_0(\mathbb{R}^N)$, $\alpha\in(0,2)$, $\beta\geq 0$, $p>1$, and $T>0$. We say that $u\in C([0,T),C_0(\mathbb{R}^N))$
is a mild solution of problem \eqref{eq} if $u$ satisfies the following integral equation
\begin{equation}\label{IE}
    u(t)=E_\alpha\left(\frac{t^{\beta+1}}{\beta+1}\right)\ast u_0(x)-\int_{0}^th(s) E_\alpha\left(\frac{t^{\beta+1}}{\beta+1}-\frac{s^{\beta+1}}{\beta+1} \right)\ast |u|^{p-1}u(x,s)\,ds,\quad t\in[0,T).
\end{equation}
More general, for all $0\leq t_0\leq t<T$, we have
\begin{equation}\label{IEG}
    u(t)=E_\alpha\left(\frac{t^{\beta+1}}{\beta+1}-\frac{t_0^{\beta+1}}{\beta+1} \right)\ast u(x,t_0)-\int_{t_0}^th(s) E_\alpha\left(\frac{t^{\beta+1}}{\beta+1}-\frac{s^{\beta+1}}{\beta+1} \right)\ast |u|^{p-1}u(x,s)\,ds.
\end{equation}
\end{definition}
We refer the reader to \cite{CH,6} to get the existence, the uniqueness and the regularity of mild solution of \eqref{eq}.
\begin{theorem}[Global existence]\label{T0}
Given $0\leq u_0\in  L^1(\mathbb{R}^N)\cap C_0(\mathbb{R}^N)$, $\alpha\in(0,2)$, $\beta\geq 0$, and $p>1$. Then, problem \eqref{eq} has a unique global mild solution 
$$u\in C([0,\infty),L^1(\mathbb{R}^N)\cap C_0(\mathbb{R}^N))\cap C^1((0,\infty),L^2(\mathbb{R}^N))\cap C((0,\infty),H^2(\mathbb{R}^N)).$$
\end{theorem}

\begin{lemma}[Nonnegativity]\label{nonnegativity}
Let $T>0$. If $u$ is a mild solution of problem \eqref{eq} on $[0,T)$, and $u_0\geq 0$, then $u(x,t)\geq 0$ for almost everywhere $x\in \mathbb{R}^N$ and for all $t\in [0,T)$.
\end{lemma}
\begin{proof}
Our goal is to prove $u^-=0$, where $u=u^+-u^-$, $u^+=\max(u,0)$, and $u^-=\max(-u,0)$. Multiplying the first equation of the system \eqref{eq} by $u^-$ and integrating over $\mathbb{R}^N$, we obtain
\begin{equation*}
\int_{\mathbb{R}^N}u_tu^- \,dx=\int_{\mathbb{R}^N}t^\beta\Delta(u)u^-\,dx -\int_{\mathbb{R}^N}t^\beta(-\Delta)^{\alpha/2}(u) u^-\,dx - \int_{\mathbb{R}^N}h(t)|u|^{p-1}u u^-\,dx.
\end{equation*}
Using the identities $u^+\, u^-=\Delta(u^+)u^-=0$, and $(-\Delta)^{\alpha/2}(u^+)u^-\leq 0$ almost everywhere, we get
\begin{equation}
-\int_{\mathbb{R}^N}u^{-}_t u^- \,dx\geq-\int_{\mathbb{R}^N}t^\beta\Delta(u^-)u^- \,dx +\int_{\mathbb{R}^N}t^\beta(-\Delta)^{\alpha/2}(u^-) u^- \,dx + \int_{\mathbb{R}^N}h(t)|u|^{p-1}(u^-)^2\,dx.
\end{equation}
Thus
\begin{equation}\label{MP1}
\frac{1}{2}\frac{d}{dt}\int_{\mathbb{R}^N}(u^{-})^2 \,dx\leq t^\beta\int_{\mathbb{R}^N}\Delta(u^-)u^- \,dx -t^\beta\int_{\mathbb{R}^N}(-\Delta)^{\alpha/2}(u^-) u^- \,dx\\
 -h(t) \int_{\mathbb{R}^N}|u|^{p-1}(u^-)^2\,dx.
\end{equation}
Applying Green's theorem, we have
\begin{equation}\label{MP2}
\int_{\mathbb{R}^N}\Delta(u^-)u^- \,dx=-\int_{\mathbb{R}^N}|\nabla(u^-)|^2\,dx,
\end{equation} 
and, using the self-adjoint property, we obtain
\begin{equation}\label{MP3}
\int_{\mathbb{R}^N}u^- (-\Delta)^{\alpha/2}u^-\,dx=\int_{\mathbb{R}^N}[(-\Delta)^{\alpha/4}(u^-)] ^2 \,dx.
\end{equation}
Inserting \eqref{MP2} and \eqref{MP3} into \eqref{MP1}, we infer that
\begin{equation}\label{MP4}
\frac{1}{2}\frac{d}{dt}\int_{\mathbb{R}^N}(u^{-})^2 \,dx\leq 0.
\end{equation}
Integrating \eqref{MP4} with respect to time, we arrive at 
\begin{equation}
\int_{\mathbb{R}^N}(u^{-}(x,t))^2 \,dx \leq \int_{\mathbb{R}^N}(u^{-}_0(x))^2 \,dx=0,\quad\hbox{for all}\,\, t\geq 0,
\end{equation}
which implies that $u^-=0$ a.e. $x\in\mathbb{R}^N$, for all $t\in[0,T)$,  and hence, $u=u^+\geq 0$.
\end{proof}
\begin{lemma}[The comparison principle]\label{Comparison}
Let $T>0$, and let $u$ and $v$, respectively, be mild solutions of problem \eqref{eq} with initial data $u_0$ and $v_0$, respectively. If $0\leq u_0\leq v_0$, then $0\leq u(x,t)\leq v(x,t)$ for almost every $x\in \mathbb{R}^N$ and for all $t\in [0,T)$.
\end{lemma}
\begin{proof}
Let $w(x,t)=v(x,t)-u(x,t)$, then
$$
\left\{\begin{array}{ll}
w_t=-t^\beta\mathcal{L} w-h(t)(v^p-u^p),&\quad x\in\mathbb{R}^N,\,t>0,\\
w(x,0)=w_0(x)\geq 0,&\quad x\in\mathbb{R}^N.
\end{array}
\right.
$$
Using the following estimation
$$0\leq v^p-u^p\leq C(v-u)(v^{p-1}+u^{p-1})\leq C (\|v\|_\infty^{p-1}+\|u\|_\infty^{p-1}) w,$$
we arrive at
$$
\left\{\begin{array}{ll}
w_t\geq -t^\beta\mathcal{L} w-C\,h(t)(\|v\|_\infty^{p-1}+\|u\|_\infty^{p-1}) w,&\quad x\in\mathbb{R}^N,\,t>0,\\
w(x,0)=w_0(x)\geq 0,&\quad x\in\mathbb{R}^N.
\end{array}
\right.
$$
Applying similar calculations as in the proof of Lemma \ref{nonnegativity}, we conclude that $w=w^+\geq 0$ a.e. $x\in\mathbb{R}^N$, for all $t\in [0,T)$,  and therefore $v\geq u$. This completes the proof of Lemma \ref{Comparison}.
\end{proof}

%%%%%%%%%%%%%%%%%%%%%%%%%%%%%%%%%%%%%%%%%%%%%%%%%%%%%%%%%%%%%%%%%%%%%%%%%%%%%%%%%%%%%%%
\section{Main results}\label{sec3}

We deal with problem \eqref{eq} and we study the decay of the ``mass''
\begin{equation}\label{mass}
    M(t)\equiv \int_{\mathbb{R}^N}u(x,t)\,dx=\int_{\mathbb{R}^N}u_0(x)\,dx -\int_0^t\int_{\mathbb{R}^N}h(s)u^p(x,s)\,dxds.
\end{equation}
In order to obtain equality \eqref{mass}, it
suffices to integrate \eqref{IE} with respect to $x$, using property $(ii)$ in Lemma \ref{Property1}, and applying the Fubini theorem.

Since we limit ourselves to non-negative solutions, the function
$M(t)$ defined in \eqref{mass} is non-negative and non-increasing. Hence, the limit
$M_\infty = \lim_{t\rightarrow\infty}M(t)$ exists and we answer
the question whether it is equal to zero or not.
In our first theorem, the  diffusion phenomena determine the large time
asymptotics of solution to \eqref{eq}.
\begin{theorem}\label{decay}
Let $u=u(x,t)$ be a non-negative nontrivial global mild solution of
$\eqref{eq}$. If
\begin{equation}\label{conditionh}
 \int_1^\infty  t^{-\frac{N}{\alpha}(p-1)(1+\beta)}h(t)dt<\infty,
\end{equation} 
then
\begin{equation}\label{2.4}
\lim_{t\to\infty}M(t)=M_\infty >0.
\end{equation}
Moreover, for all $q\in[1,\infty)$,
\begin{equation}\label{2.3}
    t^{\frac{N}{\alpha}\left(1-\frac{1}{q}\right)(1+\beta)}\left\|u(t)-M_\infty
E_\alpha\left(\frac{t^{\beta+1}}{\beta+1}\right)\right\|_q\longrightarrow
0\quad\hbox{as}\;\; t\to\infty.
\end{equation}
\end{theorem}
\begin{remark}\label{remark1} The condition \eqref{conditionh} can be replaced by
$$ \int_{t_0}^\infty  t^{-\frac{N}{\alpha}(p-1)(1+\beta)}h(t)dt<\infty,\qquad\hbox{for any}\,\,t_0>0.$$
On the other hand, if $h\in L^\infty(0,\infty)$ and $p> 1+\frac{\alpha}{N(\beta+1)}$, then the condition \eqref{conditionh} in Theorem \ref{decay} is fulfilled. Indeed, as $ p>1+\frac{\alpha}{N(\beta+1)}\Rightarrow\frac{N(p-1)(\beta+1)}{\alpha}>1$, we have
$$\int_1^\infty  t^{-\frac{N}{\alpha}(p-1)(1+\beta)}h(t)dt\leq \|h\|_\infty\int_1^\infty  t^{-\frac{N}{\alpha}(p-1)(1+\beta)}dt<\infty.$$
\end{remark}

In the remaining range of $p$, the mass $M(t)$ converges to zero and
this phenomena can be interpreted as the domination of nonlinear
effects in the large time asymptotic of solutions to \eqref{eq}. Note here that the mass $\displaystyle M(t)=\int_{\mathbb{R}^N}u(x,t)\,dx$
 of every solution to the linear equation $u_t+\mathcal{L} u=0$ is constant.

\begin{theorem}\label{convto0}
Let $u=u(x,t)$ be a non-negative global mild solution of problem
$\eqref{eq}$. If $\displaystyle \inf_{t\geq 0}h(t)>0$ and $1<p\leq 1+\frac{\alpha}{N(\beta+1)}$, then
$$
\lim_{t\to\infty}M(t)=M_\infty=0.
$$
\end{theorem}
An immediate consequence of Theorems \ref{decay} and \ref{convto0}, using Remark \ref{remark1}, is the following

\begin{theorem}\label{general}
Let $u=u(x,t)$ be a non-negative nontrivial global mild solution of problem
$\eqref{eq}$. If $h\in L^\infty(0,\infty)$ and $\displaystyle \inf_{t\geq 0}h(t)>0$, then we have two cases:
\begin{itemize}
\item If $p>1+\frac{\alpha}{N(\beta+1)}$, then $M_\infty>0$. Moreover, \eqref{2.3} holds for all $q\in[1,\infty)$.
\item If $1<p\leq 1+\frac{\alpha}{N(\beta+1)}$, then $M_\infty=0$.
\end{itemize}
\end{theorem}

%%%%%%%%%%%%%%%%%%%%%%%%%%%%%%%%%%%%%%%%%%%%%%%%%%%%%%%%%%%%%%%%%%%%%%%%%
\section{Proof of Theorem \ref{decay}}\label{sec4}
We start by proving $M_\infty>0$. From \eqref{mass}, we have
\begin{equation}\label{3.1}
    0\leq\int_{\mathbb{R}^N}u(x,t)\,dx=\int_{\mathbb{R}^N}u_0(x)\,dx - \int_0^t\int_{\mathbb{R}^N}h(s)u^p(x,s)\,dx\,ds.
\end{equation}
 Hence, for $u_0\in L^1(\mathbb{R}^N)$, we immediately obtain
\begin{equation}\label{3.2}
u\in L^\infty((0,\infty),L^1(\mathbb{R}^N))\quad\hbox{and}\quad  \int_0^\infty\int_{\mathbb{R}^N}h(s)u^p(x,t)\,dx\,dt\leq \|u_0\|_1<\infty.
\end{equation}
As $u$ is a mild solution, we immediately get the following estimate
$$
0\leq u(t)\leq E_\alpha\left(\frac{t^{\beta+1}}{\beta+1}\right)\ast u_0(x),\qquad\hbox{for all}\,\,t> 0,
$$
which implies, using \eqref{Pr1}, for all $t> 0$, that
\begin{eqnarray}\label{HTP}
  \|u(t)\|^p_p &\leq&\left\|E_\alpha\left(\frac{t^{\beta+1}}{\beta+1}\right)\ast u_0\right\|^p_p\nonumber\\
  &\leq& \min\left\{C\,t^{-\frac{N(\beta+1)(p-1)}{2}}\|u_0\|^p_1,\,C\,t^{-\frac{N(\beta+1)(p-1)}{\alpha}}\|u_0\|^p_1,\,\|u_0\|^p_p\right\}\nonumber\\
  &=&H(t,p,\alpha,\beta,u_0).
\end{eqnarray}
Now, for fixed $\varepsilon\in(0,1]$, we denote by $u^\varepsilon=u^\varepsilon(x,t)$ the solution of \eqref{eq} with initial condition $\varepsilon u_0$. The comparison principle implies that
$$0\leq u^\varepsilon(x,t)\leq u(x,t),\qquad\hbox{for all}\,\,x\in\mathbb{R}^N\,\,t>0.$$
Hence,
$$M_\infty\geq M^\varepsilon_\infty\equiv\lim\limits_{t\rightarrow\infty}\int_{\mathbb{R}^N}u^\varepsilon(x,t)\,dx .$$
Therefore, to prove \eqref{2.4}, it suffices to show, for small $\varepsilon>0$ which will be determined later, that $M^\varepsilon_\infty >0$.
Applying \eqref{3.1} to $u^\varepsilon$ and letting $t\rightarrow\infty$, we obtain
\begin{equation}\label{3.11}
  M^\varepsilon_\infty = \varepsilon\left\{\int_{\mathbb{R}^N}u_0(x)\,dx -\frac{1}{\varepsilon} \int_0^\infty\int_{\mathbb{R}^N}h(t)\left(u^\varepsilon(x,t)\right)^p\,dx\,dt\right\}.
\end{equation}
Similarly, by applying \eqref{HTP} to $u^\varepsilon$, we get
$$ \|u^\varepsilon(t)\|^p_p\leq H(t,p,\alpha,\beta,\varepsilon u_0)=\varepsilon^pH(t,p,\alpha,\beta,u_0).$$
Hence
\begin{eqnarray*}
  \frac{1}{\varepsilon}\int_0^\infty\int_{\mathbb{R}^N}h(t)\left(u^\varepsilon (x,t)\right)^p\,dx\,dt&=& \frac{1}{\varepsilon}\int_0^\infty h(t)  \|u^\varepsilon(t)\|^p_p\,dt\\
&\leq&\varepsilon^{p-1}\int_0^\infty h(t)H(t,p,\alpha,\beta,u_0)\,dt.
\end{eqnarray*}
 It follows immediately from the definition of the function $H$ that
\begin{eqnarray*}
\int_{0}^{1}h(t)H(t,p,\alpha,\beta,u_0)\,dt&\leq& \|u_0\|_p^p\int_0^1h(t)dt<\infty,\\
\int_{1}^{\infty}h(t)H(t,p,\alpha,\beta,u_0)\,dt&\leq& C\,\|u_0\|_1^p\int_1^{\infty}\min\left\{C\,t^{-\frac{N(\beta+1)(p-1)}{2}},\,C\,t^{-\frac{N(\beta+1)(p-1)}{\alpha}}\right\}h(t)dt\\
&=& C\,\|u_0\|_1^p\int_1^{\infty}t^{-\frac{N(\beta+1)(p-1)}{\alpha}}h(t)dt<\infty.
\end{eqnarray*}
Consequently,
$$\int_{0}^{\infty}h(t)H(t,p,\alpha,\beta,u_0)\,dt<\infty,$$
which yields to
$$
\lim_{\varepsilon\rightarrow 0^+}\frac{1}{\varepsilon} \int_0^\infty\int_{\mathbb{R}^N}\left(u^\varepsilon(x,t)\right)^p\,dx\,dt=0.
$$
Hence, as $\displaystyle \int_{\mathbb{R}^N}u_0(x)\,dx>0$, there exists $\varepsilon_0\in(0,1)$ such that
$$\frac{1}{\varepsilon_0} \int_0^\infty\int_{\mathbb{R}^N}\left(u^{\varepsilon_0}(x,t)\right)^p\,dx\,dt\leq \frac{1}{2}\int_{\mathbb{R}^N}u_0(x)\,dx,$$
which implies that
$$M_\infty\geq M^{\varepsilon_0}_\infty\geq \frac{\varepsilon_0}{2}\int_{\mathbb{R}^N}u_0(x)\,dx>0.$$
This completes the proof of \eqref{2.4}. The proof of the asymptotic behavior \eqref{2.3} will be divided into two cases.\\

\noindent Case of $p=1$. Applying Minkowski's inequality, we have
\begin{eqnarray}\label{4.1}
\left\|u(t)-M_\infty E_\alpha\left(\frac{t^{\beta+1}}{\beta+1}\right)\right\|_1 &\leq& \left\|u(t)-E_\alpha\left(\frac{t^{\beta+1}}{\beta+1}-\frac{t_0^{\beta+1}}{\beta+1}\right)\ast u(t_0)\right\|_1\nonumber\\
&{}&+\left\|E_\alpha\left(\frac{t^{\beta+1}}{\beta+1}-\frac{t_0^{\beta+1}}{\beta+1}\right)\ast u(t_0)-M(t_0)E_\alpha\left(\frac{t^{\beta+1}}{\beta+1}-\frac{t_0^{\beta+1}}{\beta+1}\right)\right\|_1\nonumber\\
&{}&+\left\|M(t_0)E_\alpha\left(\frac{t^{\beta+1}}{\beta+1}-\frac{t_0^{\beta+1}}{\beta+1}\right)-M(t_0)E_\alpha\left(\frac{t^{\beta+1}}{\beta+1}\right)\right\|_1\nonumber\\
&{}&+\left\|M(t_0)E_\alpha\left(\frac{t^{\beta+1}}{\beta+1}\right)-M_\infty E_\alpha\left(\frac{t^{\beta+1}}{\beta+1}\right)\right\|_1,
\end{eqnarray}
for all $t\geq t_0 \geq 0$. As $u$ is a mild solution, using \eqref{Pr1}, we get from \eqref{IEG} that
\begin{equation}\label{4.2}
\left\|u(t)-E_\alpha\left(\frac{t^{\beta+1}}{\beta+1}-\frac{t_0^{\beta+1}}{\beta+1}\right)\ast u(t_0)\right\|_1\leq C\int_{t_0}^t h(s)\|u(s)\|_p^p ds,
\end{equation} 
for all $t\geq t_0 \geq 0$. On the other hand, applying Lemma \ref{Taylor} with $g=u_0$, we get
\begin{equation}\label{4.3}
\lim_{t\rightarrow\infty}\left\|E_\alpha\left(\frac{t^{\beta+1}}{\beta+1}-\frac{t_0^{\beta+1}}{\beta+1}\right)\ast u(t_0)-M(t_0)E_\alpha\left(\frac{t^{\beta+1}}{\beta+1}-\frac{t_0^{\beta+1}}{\beta+1}\right)\right\|_1=0.
\end{equation} 
Applying again Lemma \ref{Taylor} with $g=E_\alpha\left(\frac{t_0^{\beta+1}}{\beta+1}\right)$, using $(ii)$ Lemma \ref{Property1}, we obtain
$$
\lim_{t\rightarrow\infty}\left\|E_\alpha\left(\frac{t^{\beta+1}}{\beta+1}-\frac{t_0^{\beta+1}}{\beta+1}\right)-E_\alpha\left(\frac{t^{\beta+1}}{\beta+1}-\frac{t_0^{\beta+1}}{\beta+1}\right)\ast E_\alpha\left(\frac{t_0^{\beta+1}}{\beta+1}\right)\right\|_1=0,
$$ 
which implies, using $(iii)$ Lemma \ref{Property1}, that
\begin{eqnarray}\label{4.4}
&{}&\left\|M(t_0)E_\alpha\left(\frac{t^{\beta+1}}{\beta+1}-\frac{t_0^{\beta+1}}{\beta+1}\right)-M(t_0)E_\alpha\left(\frac{t^{\beta+1}}{\beta+1}\right)\right\|_1\nonumber\\
&{}&\leq M(0)\left\|E_\alpha\left(\frac{t^{\beta+1}}{\beta+1}-\frac{t_0^{\beta+1}}{\beta+1}\right)-E_\alpha\left(\frac{t^{\beta+1}}{\beta+1}\right)\right\|_1\nonumber\\
&{}&= M(0)\left\|E_\alpha\left(\frac{t^{\beta+1}}{\beta+1}-\frac{t_0^{\beta+1}}{\beta+1}\right)-E_\alpha\left(\frac{t^{\beta+1}}{\beta+1}-\frac{t_0^{\beta+1}}{\beta+1}\right)\ast E_\alpha\left(\frac{t_0^{\beta+1}}{\beta+1}\right)\right\|_1\longrightarrow 0,
\end{eqnarray}
when $t$ goes to $\infty$. In addition, using $(ii)$ Lemma \ref{Property1}, we have
\begin{equation}\label{4.5}
\left\|M(t_0)E_\alpha\left(\frac{t^{\beta+1}}{\beta+1}\right)-M_\infty E_\alpha\left(\frac{t^{\beta+1}}{\beta+1}\right)\right\|_1\leq |M(t_0)-M_\infty|.
\end{equation} 

Therefore, inserting \eqref{4.2}-\eqref{4.5} into \eqref{4.1}, we arrive at
\begin{equation}
\limsup_{t\rightarrow\infty}\left\|u(t)-M_\infty E_\alpha\left(\frac{t^{\beta+1}}{\beta+1}\right)\right\|_1\leq C\int_{t_0}^\infty h(s)\|u(t)\|_p^p ds+|M(t_0)-M_\infty|.
\end{equation}
Letting $t_0\to 0$, it follows that 
\begin{equation}
\lim\limits_{t\to\infty}\left\|u(t)-M_\infty E_\alpha\left(\frac{t^{\beta+1}}{\beta+1}\right)\right\|_1=0.
\end{equation}
\noindent Case of $p>1$. On the one hand, for each $m\in[1,+\infty]$, using \eqref{Pr1} and \eqref{P_2} respectively, we get
\begin{equation}\label{E1}
\|u(t)\|_m\leq\left\|E_\alpha\left(\frac{t^{\beta+1}}{\beta+1}\right) \ast u_0\right\|_m \leq  Ct^{-\frac{N}{\alpha}\left(1-\frac{1}{m}\right)(1+\beta)}\|u_0\|_1,
\end{equation}
and
\begin{eqnarray}\label{E2}
\left\|E_\alpha\left(\frac{t^{\beta+1}}{\beta+1}\right)\right\|_m&=& \left\|P_\alpha\left(\frac{t^{\beta+1}}{\beta+1}\right)\ast P_2\left(\frac{t^{\beta+1}}{\beta+1}\right)\right\|_m\nonumber\\
&\leq& C\,t^{-\frac{N}{\alpha}\left(1-\frac{1}{m}\right)(1+\beta)} \left\| P_2\left(\frac{t^{\beta+1}}{\beta+1}\right)\right\|_1\nonumber\\
&=&C\,t^{-\frac{N(m-1)}{\alpha m}(1+\beta)},
\end{eqnarray}
where we have used the fact that $\|P_2(t)\|_1=1$, for all $t>0$. On the other hand, for a fixed $m\in[1,+\infty]$, for every $q\in [1,m)$, by applying the Minkowski and H\"older inequalities, we obtain
\begin{equation}\label{final}
\left\|u(t)-M_\infty E_\alpha\left(\frac{t^{\beta+1}}{\beta+1}\right)\right\|_q\leq \left\|u(t)-M_\infty E_\alpha\left(\frac{t^{\beta+1}}{\beta+1}\right)\right\|_1^{1-\delta}\left(\|u(t)\|_m^\delta+\left\|M_\infty E_\alpha\left(\frac{t^{\beta+1}}{\beta+1}\right)\right\|_m^\delta\right)
\end{equation}  
where $\delta=(1-1/q)/(1-1/m)$. Inserting \eqref{E1} and \eqref{E2} into \eqref{final}, we get
$$
t^{\frac{N}{\alpha}\left(1-\frac{1}{q}\right)(1+\beta)}\left\|u(t)-M_\infty
P\left(\frac{t^{\beta+1}}{\beta+1}\right)\right\|_q\leq C \left\|u(t)-M_\infty E_\alpha\left(\frac{t^{\beta+1}}{\beta+1}\right)\right\|_1^{1-\delta}\longrightarrow 0,\quad\hbox{as}\;\; t\to\infty.
$$
This completes the proof of Theorem \ref{decay}.\hfill$\square$

%%%%%%%%%%%%%%%%%%%%%%%%%%%%%%%%%%%%%%%%%%%%%%%%%%%%%%%%%%%%%%%%%%%%%%%%%%%%%%%%%%%%%%%

\section{Proof of Theorem \ref{convto0}}\label{sec5}
The proof of Theorem \ref{convto0} is based on the rescaled test function method. Let $\varphi\in C_c^\infty([0,\infty)\times \mathbb{R}^N)$ be such that
$$\varphi(x,t)=\varphi_1(x)(\varphi_2(t))^\ell,\qquad \ell=\frac{2p-1}{p-1},$$
where 
$$ \varphi_1(x) =\langle {x}/{BR}\rangle^{-q_0}=\left(1+\left|\frac{x}{BR}\right|^2\right)^{-q_0/2},\qquad  \varphi_2(t) = \psi\left(\frac{t^{\beta+1}}{R^\alpha}\right),\qquad R>0,
$$
with $N<q_0<N+\alpha p$, and
\begin{equation}\label{psi}
\psi(r)=
\left\{
\begin{array}{ll}
 1&\,\, \text{if} \quad 0\leq r \leq 1, \\
 \searrow &\,\, \text{if} \quad 1\leq r \leq 2, \\
 0 &\,\, \text{if} \quad r\geq 2.
\end{array}
\right.
\end{equation}
The constant $B>0$ in the definition of $\varphi_1$ is fixed and will be chosen later. Since $u$ is a mild solution, we get 
\begin{eqnarray}\label{5.1}
&{}&\int_{\mathbb{R}^N}u_0(x)\varphi_1(x)\,dx-\int_{\Omega}\int_{\mathbb{R}^N}h(t)u^p(x,t)\varphi(x,t)\,dx\,dt\nonumber\\
&{}&\,=\int_{\Omega}\int_{\mathbb{R}^N}t^\beta (\varphi_2(t))^\ell u(x,t)\mathcal{L}(\varphi_1(x)) \,dx\,dt-\ell \int_{\Omega}\int_{\mathbb{R}^N}u(x,t)(\varphi_2(t))^{\ell-1} \varphi_1(x) \varphi'_2(t)\,dx\,dt\nonumber\\
&{}&\leq \int_{\Omega}\int_{\mathbb{R}^N}u(x,t) \,t^\beta (\varphi_2(t))^\ell \left|\mathcal{L}(\varphi_1(x))\right| \,dx\,dt + \ell \int_{\Omega}\int_{\mathbb{R}^N}u(x,t)\,\varphi_1(x)(\varphi_2(t))^{\ell-1}\left| \varphi'_2(t)\right|\,dx\,dt\nonumber\\
&{}&=: I_1+I_2,
\end{eqnarray}
where $\Omega=\{t\in[0,\infty);\, t\leq (2R^\alpha)^{\frac{1}{\beta+1}}\}$. Applying the $\varepsilon-$Young inequality,
$$ab\leq \varepsilon a^p+C_\varepsilon b^q,\quad\hbox{whith}\,\, q=p/(p-1)=\ell-1,$$
we get
\begin{eqnarray}\label{I1}
I_1&=& \int_{\Omega}\int_{\mathbb{R}^N}u(x,t)(\varphi(t))^{1/p} (\varphi(t))^{-1/p} \,t^\beta (\varphi_2(t))^\ell \left|\mathcal{L}(\varphi_1(x))\right| \,dx\,dt\nonumber\\
&\leq& \varepsilon\int_{\Omega}\int_{\mathbb{R}^N}u^p(x,t)\varphi(t)\,dx\,dt+ C_\varepsilon\int_{\Omega}\int_{\mathbb{R}^N}t^{\frac{\beta p}{p-1}} (\varphi_2(t))^{\ell}(\varphi_1(x))^{-1/(p-1)}\left|\mathcal{L}\varphi_1(x)\right|^{p/(p-1)} \,dx\,dt,
\end{eqnarray}
and 
\begin{eqnarray}\label{I2}
I_2&=& \ell \int_{\Omega}\int_{\mathbb{R}^N}u(x,t)(\varphi(t))^{1/p} (\varphi(t))^{-1/p}\varphi_1(x)(\varphi_2(t))^{\ell-1}\left|\varphi'_2(t)\right|\,dx\,dt\nonumber\\
&\leq& \ell\varepsilon\int_{\Omega}\int_{\mathbb{R}^N}u^p(x,t)\varphi(t)\,dx\,dt+ \ell \,C_\varepsilon\int_{\Omega}\int_{\mathbb{R}^N} \varphi_1(x)\varphi_2(t)\left|\varphi'_2(t)\right|^{\frac{p}{p-1}} \,dx\,dt.
\end{eqnarray}
Combining \eqref{I1}, \eqref{I2} and \eqref{5.1}, we arrive at
\begin{eqnarray}\label{5.2}
&{}&\int_{\mathbb{R}^N}u_0(x)\varphi_1(x)\,dx-\int_{\Omega}\int_{\mathbb{R}^N}(h(t)+ (1+\ell)\varepsilon)u^p(x,t)\varphi(x,t)\,dx\,dt\nonumber\\
&{}&\leq C_\varepsilon\int_{\Omega}\int_{\mathbb{R}^N}t^{\frac{\beta p}{p-1}} (\varphi_2(t))^{\ell}(\varphi_1(x))^{-1/(p-1)}\left|\mathcal{L}\varphi_1(x)\right|^{p/(p-1)} \,dx\,dt\nonumber\\
&{}&\quad +\quad  \ell \,C_\varepsilon\int_{\Omega}\int_{\mathbb{R}^N} \varphi_1(x)\varphi_2(t)\left|\varphi'_2(t)\right|^{\frac{p}{p-1}} \,dx\,dt\nonumber\\
&{}&=: J_1+J_2.
\end{eqnarray}
Using Lemma \ref{lemma5}, by choosing $R,B\geq 1$, we obtain
\begin{eqnarray*}
J_1&=&C_\varepsilon\int_{\Omega}t^{\frac{\beta p}{p-1}} (\varphi_2(t))^{\ell}\,dt\int_{\mathbb{R}^N}(\varphi_1(x))^{-1/(p-1)}\left|\mathcal{L}\varphi_1(x)\right|^{p/(p-1)} \,dx\\
&\leq& C\,B^{-\frac{\alpha p}{p-1}+N}\,R^{-\frac{\alpha p}{p-1}+N}\int_{\Omega}t^{\frac{\beta p}{p-1}} (\varphi_2(t))^{\ell}\,dt.
\end{eqnarray*}
By the change of variable $\eta =t^{\beta+1}R^{-\alpha}$, and the fact that $\psi\leq 1$, we get
\begin{eqnarray}\label{5.3}
J_1&\leq&C\,B^{-\frac{\alpha p}{p-1}+N}\,R^{-\frac{\alpha }{(\beta+1)(p-1)}+N}\int_0^2 \eta^{\frac{\beta }{(\beta+1)(p-1)}} (\psi(\eta))^{\ell}\,d\eta\leq C\,B^{-\frac{\alpha p}{p-1}+N}\,R^{-\frac{\alpha }{(\beta+1)(p-1)}+N}.
\end{eqnarray}
On the other hand, using the change of variable $y=(BR)^{-1}\,x$, we have
\begin{eqnarray*}
J_2&=&\ell\,C_\varepsilon\int_{\mathbb{R}^N} \varphi_1(x)\,dx \int_{\Omega}\varphi_2(t)\left|\varphi'_2(t)\right|^{\frac{p}{p-1}} \,dt\\
&=& C\,(BR)^{N}\int_{\mathbb{R}^N} \langle y\rangle^{-q_0}\,dy\int_{\Omega}\varphi_2(t)\left|\varphi'_2(t)\right|^{\frac{p}{p-1}} \,dt.
\end{eqnarray*}
By the change of variable $\eta =t^{\beta+1}R^{-\alpha}$, and $q_0>N$, we get
\begin{eqnarray}\label{5.4}
J_2&\leq&C\,B^{N}R^{-\frac{\alpha }{(\beta+1)(p-1)}+N}\int_0^2 \eta^{\frac{\beta }{(\beta+1)(p-1)}} \psi(\eta)\left|\psi'(\eta)\right|^{\frac{p}{p-1}} \,d\eta \leq C\,R^{-\frac{\alpha }{(\beta+1)(p-1)}+N}\,B^N.
\end{eqnarray}
Using \eqref{5.3} and \eqref{5.4} into \eqref{5.2}, we conclude that
\begin{equation}\label{5.5}
\int_{\mathbb{R}^N}u_0(x)\varphi_1(x)\,dx-\int_{\Omega}\int_{\mathbb{R}^N}(h(t)+ (1+\ell)\varepsilon)u^p(x,t)\varphi(x,t)\,dx\,dt\leq C\left(B^{N-\frac{\alpha p}{p-1}}+B^N\right)R^{N-\frac{\alpha }{(\beta+1)(p-1)}}.
\end{equation}
As
$$N-\frac{\alpha }{(\beta+1)(p-1)}\leq 0\quad \Longleftrightarrow\quad p\leq 1+\frac{\alpha}{N(\beta+1)},$$
we have two cases to distinguish.\\

{\bf Case 1:} $p< 1+\alpha/N(\beta+1)$. By choosing $B=1$, letting $R\rightarrow\infty$, and using the Lebesgue dominated convergence theorem in \eqref{5.5}, we get
\begin{equation}\label{5.6}
M_\infty=\int_{\mathbb{R}^N}u_0(x)\,dx-\int_0^\infty\int_{\mathbb{R}^N}h(t)u^p(x,t)\,dx\,dt\leq (1+\ell)\varepsilon\int_0^\infty\int_{\mathbb{R}^N}u^p(x,t)\,dx\,dt.
\end{equation}
Using \eqref{3.2} and the fact that $\displaystyle \inf_{t\geq 0}h(t)>0$, the integral term in the right hand-side in \eqref{5.6} is bounded and therefore
$$M_\infty\leq C\,\varepsilon.$$
Since $\varepsilon> 0$ can be chosen arbitrary small, we immediately conclude that $M_\infty= 0$.\\

{\bf Case 2:}  $p= 1+\alpha/N(\beta+1)$.  We estimate $I_2$ in \eqref{5.1}  by using H\"{o}lder's inequality instead of $\varepsilon-$Young's inequality; we obtain
\begin{eqnarray*}
I_2&=& \ell \int_{\Omega}\int_{\mathbb{R}^N}u(x,t)(\varphi(t))^{1/p} (\varphi(t))^{-1/p}\varphi_1(x)(\varphi_2(t))^{\ell-1}\left|\varphi'_2(t)\right|\,dx\,dt\nonumber\\
&\leq& \ell\left(\int_{\widetilde{\Omega}}\int_{\mathbb{R}^N}u^p(x,t)\varphi(t)\,dx\,dt\right)^{1/p}\left(\int_{\Omega}\int_{\mathbb{R}^N} \varphi_1(x)\varphi_2(t)\left|\varphi'_2(t)\right|^{\frac{p}{p-1}} \,dx\,dt\right)^{(p-1)/p}\nonumber\\
&=& C\,J_2^{(p-1)/p}\left(\int_{\widetilde{\Omega}}\int_{\mathbb{R}^N}u^p(x,t)\varphi(t)\,dx\,dt\right)^{1/p},
\end{eqnarray*}
where $J_2$ is defined in \eqref{5.2}, and $\widetilde{\Omega}=\{t\in[0,\infty);\, R^{\frac{\alpha}{\beta+1}}\leq t\leq (2R^\alpha)^{\frac{1}{\beta+1}}\}$ is the support of $\varphi'_2(t)$. Therefore, by using \eqref{I1}, we conclude from \eqref{5.1} that
\begin{eqnarray}\label{5.7}
&{}&\int_{\mathbb{R}^N}u_0(x)\varphi_1(x)\,dx-\int_{\Omega}\int_{\mathbb{R}^N}(h(t)+ \varepsilon)u^p(x,t)\varphi(x,t)\,dx\,dt\nonumber\\
&{}&\leq J_1+C\,J_2^{(p-1)/p}\left(\int_{\widetilde{\Omega}}\int_{\mathbb{R}^N}u^p(x,t)\varphi(t)\,dx\,dt\right)^{1/p},
\end{eqnarray}
where $J_1$ is defined in \eqref{5.2}. Using \eqref{5.3} and \eqref{5.4} into \eqref{5.7}, we arrive at
\begin{eqnarray}\label{5.8}
&{}&\int_{\mathbb{R}^N}u_0(x)\varphi_1(x)\,dx-\int_{\Omega}\int_{\mathbb{R}^N}(h(t)+ \varepsilon)u^p(x,t)\varphi(x,t)\,dx\,dt\nonumber\\
&&\leq C\,B^{N-\frac{\alpha p}{p-1}}\,R^{N-\frac{\alpha }{(\beta+1)(p-1)}}+C\,R^{\frac{N(p-1)}{p}-\frac{\alpha}{(\beta+1)p}}\,B^\frac{N(p-1)}{p} \left(\int_{\widetilde{\Omega}}\int_{\mathbb{R}^N}u^p(x,t)\varphi(t)\,dx\,dt\right)^{1/p}.
\end{eqnarray}
Note that, as \eqref{3.2} and $\displaystyle \inf_{t\geq 0}h(t)>0$ imply $u\in L^p((0,\infty),L^p(\mathbb{R}^N))$, we can see that
$$
\int_{\widetilde{\Omega}}\int_{\mathbb{R}^N}u^p(x,t)\varphi(t)\,dx\,dt\longrightarrow 0
\quad\hbox{as}\quad R\to\infty.$$
Therefore, by letting $R\to\infty$ in \eqref{5.8} and using the dominated convergence theorem, we infer that
\begin{equation}\label{5.9}
M_\infty-\varepsilon\int_0^\infty\int_{\mathbb{R}^N} u^p(x,t)\,dx\,dt \leq C\,B^{N-\frac{\alpha p}{p-1}}.
\end{equation}
At this stage, letting $B\to\infty$ and using again $u\in L^p((0,\infty),L^p(\mathbb{R}^N))$, we get
$$M_\infty\leq \varepsilon\int_0^\infty\int_{\mathbb{R}^N} u^p(x,t)\,dx\,dt\leq C\,\varepsilon,$$
which implies that $M_\infty= 0$ since $\varepsilon> 0$ can be arbitrary small. This completes the proof of Theorem \ref{convto0}. \hfill$\square$

%%%%%%%%%%%%%%%%%%%%%%%%%%%%%%%%%%%%%%%%%%%%%%%%%%%%%%%%%%%%%%%%%%%%%%%%%%%%%%%%%%%%%%%

\end{document}